\documentclass{SCAEOL}
\numberwithin{equation}{section}
\newcommand{\norm}[1]{\Vert #1 \Vert}
\newcommand{\eps}{\epsilon}

\begin{document}

\Year{2013} %
\Month{January}
\Vol{56} %
\No{1} %
\BeginPage{1} %
\EndPage{XX} %
\AuthorMark{First1 L N {\it et al.}}

\title{Two trust region type algorithms for solving nonconvex-strongly concave minimax problems\thanks{This work is supported by National Natural Science Foundation of China under the grants 12071279.}}{}

\author[1]{Tongliang Yao}{}
\author[1,2]{Zi Xu}{Corresponding author.}

\address[{\rm1}]{Department of Mathematics, Shanghai University, Shanghai {\rm 200444}, China}
\address[{\rm2}]{Newtouch Center for Mathematics of Shanghai University, Shanghai University, Shanghai {\rm 200444}, China}
\Emails{
	xuzi@shu.edu.cn}\maketitle

 {\begin{center}
		\parbox{14.5cm}{\begin{abstract}
				In this paper, we propose a Minimax Trust Region (MINIMAX-TR) algorithm and a Minimax Trust Region Algorithm with Contractions and Expansions(MINIMAX-TRACE) algorithm for solving nonconvex-strongly concave minimax problems. Both algorithms can find an $(\eps, \sqrt{\eps})$-second order stationary point(SSP) within $\mathcal{O}(\eps^{-1.5})$ iterations, which matches the best well known iteration complexity.\vspace{-3mm}
\end{abstract}}\end{center}}




\renewcommand{\baselinestretch}{1.2}

\baselineskip 11pt\parindent=10.8pt  \wuhao

\section{Introduction}
We consider the following minimax optimization problems,
\begin{equation}\label{problem}
\min \limits_{x\in \mathbb{R}^n}\max \limits_{y\in \mathbb{R}^m}f(x,y),
\end{equation}
where $f(x,y)$ is $\ell$-smooth, $\mu$-strongly concave in $y$ but possibly nonconvex in $x$. Recently, minimax problem \eqref{problem} has attracted significant attentions due to its application in machine learning and data science, such as generative adversarial networks (GANs) \cite{goodfellow2014generative}, reinforcement learning \cite{dai2018sbeed}, robust learning \cite{qian2019robust} and adversarial training \cite{sinha2018certifiable}.

In the existing literature, many first-order algorithms have been developed to solve minimax problems under the nonconvex-strongly concave setting(\cite{huang2022accelerated,GDA,lu2020hybrid,jin2020local,rafique2022weakly}). All these algorithms can find an $\eps$-first order stationary point(FSP) of $f$ or $P(x) := \max_{y\in \mathcal{Y}}f(x,y)$ with a iteration complexity of $\tilde{\mathcal{O}}\left( \kappa_y^2\eps ^{-2} \right)$. 
The gradient complexity bound can be improved to $\tilde{\mathcal{O}}\left( \sqrt{\kappa_y}\eps^{-2} \right)$ by an 
accelerated algorithm \cite{lin2020near}, which match the gradient oracle lower bound for finding $\eps$-FSP \cite{li2021complexity}.
However, FSP could be saddle points and thus these algorithms cannot guarantee the local optimality. Inspired by the classic cubic regularization algorithm, Luo et al. \cite{luocubic} proposed Minimax Cubic Newton(MCN) algorithm. The algorithm uses gradient ascent to update y, which is then used to estimate the gradient and Hessian involved in the cubic regularization update for x. MCN could find a $(\eps, \sqrt{\eps})$-SSP with the iteration complexity of $\mathcal{O}(\eps^{-1.5})$. In a concurrent work, Chen et al. \cite{chencubic} proposed a similar algorithm to MCN but with a different termination criterion. The algorithm can achieve the same iteration complexity as MCN.

\subsection{Contributions}
In this paper, we consider nonconvex-strongly concave minimax problems. We focus on finding $(\eps, \sqrt{\eps})$-SSP of $P(x)$. Inspired by the stochastic trust region algorithm\cite{stochasticTR}, we propose a Minimax Trust Region(MINIMAX-TR) algorithm. We prove that MINIMAX-TR can find an $(\eps, \sqrt{\eps})$-SSP with $\mathcal{O}(\eps^{-1.5})$ number of iterations, matching the $\mathcal{O}(\eps^{-1.5})$ complexity in (\cite{luocubic},\cite{chencubic}). MINIMAX-TR can be viewed as inexact trust region method with fixed trust region radius. MINIMAX-TR may have poor numerical performance because the trust region radius is set to a small value. To overcome this shortcoming, we propose a Minimax Trust Region Algorithm with Contractions and Expansions(MINIMAX-TRACE) algorithm inspired by \cite{TRACE}. MINIMAX-TRACE also can find an $(\eps, \sqrt{\eps})$-SSP within $\mathcal{O}(\eps^{-1.5})$ number of iterations. Under mild assumptions, we also prove that MINIMAX-TRACE has local quadratic convergence. 

\subsection{Related Works}
\textbf{Cubic regularization(CR).} Cubic regularization method was first proposed by Griewank(in an unpublished technical report \cite{griewankcubic}). In \cite{nesterov2006cubic}, Nesterov and Polyak proved that CR can converge to $(\eps, \sqrt{\eps})$-SSP with the iteration complexity of $\mathcal{O}(\eps^{-1.5})$. Cartis et al. \cite{cartis2011adaptive1,cartis2011adaptive2} proposed an adaptive cubic
regularization (ARC) approach to boost the efficiency by adaptively tunes the regularization parameter. In large machine learning problem, caluating the full gradient and hessian is a big overhead. To lower this computational cost, some stochastic or sub-sampling cubic regularization were proposed (\cite{kohler2017sub,xu2020newton,tripuraneni2018stochastic,yao2018inexact,chen2022accelerating}).

\noindent\textbf{Trust Region(TR).} Trust Region Method \cite{conn2000trust} is a classical second-order optimization algorithm.
It is known that traditional trust region methods converges at the rate $\mathcal{O}(\eps^{-2})$, Curtis et al. \cite{TRACE} proposes Trust Region Algorithm with Contractions and Expansions(TRACE) which converges at the optimal rate $\mathcal{O}(\eps^{-3/2})$. TRACE achieve the same iteration complexity as cubic regularization method. In \cite{stochasticTR}, authors proposed a novel trust region method with fixed trust region radius, it can also converge at the optimal rate $\mathcal{O}(\eps^{-3/2})$. Similar to cubic regularization method, some stochastic or sub-sampling trust region method were proposed(\cite{wang2022stochastic, chauhan2020stochastic, erway2020trust, stochasticTR, yao2018inexact, xu2020newton}) .

\noindent\textbf{Minimax problem.} We give a brief review on first-order algorithms for solving minimax optimization problems. Many algorithms have been proposed to solve convex-concave minimax optimization problems. Nemirovski \cite{nemirovski2004prox} proposed mirror-prox algorithm which converges at the rate $\mathcal{O}(\eps^{-1})$ in terms of duality gap. Nesterov \cite{nesterov2007dual} proposed dual extrapolation algorithm, which achieves a iteration complexity of $\mathcal{O}(\eps^{-1})$. Lin et al. \cite{lin2020near} proposed near-optimal algorithm with the same $\mathcal{O}(\eps^{-1})$ rate. These methods match the lower bound of \cite{ouyang2021lower}. For nonconvex-strongly concave minimax problem, Lin et al. \cite{lin2020near} proposed an accelerated algorithms called MINIMAX-PPA with a complexity of $\tilde{\mathcal{O}}(\sqrt{\kappa_y} \eps^{-2})$, which match the lower bound of   \cite{li2021complexity}. 

For general nonconvex-concave minimax problem, several multi-loop algorithms have been proposed in (\cite{lu2020hybrid},\cite{nouiehed2019solving},\cite{ostrovskii2021efficient},\cite{thekumparampil2019efficient},\cite{lin2020near},\cite{kong2021accelerated}). 
MINIMAX-PPA proposed by Lin et al. \cite{lin2020near} has the best known iteration complexity of $\mathcal{O}(\eps^{-2.5})$. Because multi-loop algorithms are relatively complicated to be implemented, some studies focus on single-loop algorithms for nonconvex-concave minimax problems. 
GDA-type algorithm (\cite{pan2021efficient}, \cite{lu2020hybrid}, \cite{GDA}, \cite{boct2023alternating}, \cite{AGP})  is a class of simple and popular algorithm, which alternates between a gradient descent update on x and a gradient ascent update on y in each iteration. Xu et al. \cite{AGP} proposed
a unified single-loop alternating gradient projection (AGP) algorithm for solving
nonconvex-concave and convex-nonconcave minimax problems, which can find an
$\eps$-stationary point with the iteration complexity of $\mathcal{O}(\eps^{-4})$. 

However, all previous works targeted finding $\eps$-FSP of $f$ or $P(x)$, which could be saddle points. To overcome this shortcoming, (\cite{luocubic}, \cite{chencubic}) proposed cubic regularized GDA, a second-order algorithm that converges to $\eps$-SSP. Cubic regularized GDA run cubic regularization update on x and maximize the objective on y alternatively.

For the black-box minimax problems where we can only access function values, some zeroth-order algorithms(\cite{xu2021zeroth}, \cite{shen2022zeroth}, \cite{wang2022zeroth}, \cite{huang2022accelerated}, \cite{sadiev2021zeroth}) for solving minimax optimization problem were proposed.

\section{Preliminaries}
In this section, we introduce some notations, definitions and propositions. 

We use $\norm{\cdot}$ to denote the spectral norm of matrices and Euclidean norm of vectors. We use $\mathbb{N}$ to denote the set
of nonnegative integers.
For a twice differentiable function $f(x,y)$, we use $\nabla_x f, \nabla_y f$ and $\nabla_{xx}^2 f, \nabla_{xy}^2 f, \nabla_{yy}^2 f$ to denote its partial gradient and partial Hessian.
Given a discrete set $\mathcal{S}$, we denote its cardinality by $|\mathcal{S}|$. Denote $P(x) = \max_{y\in \mathbb{R}^m}f(x,y)$ and $y^*(x)=\mathop{\arg\max}_{y\in\mathbb{R}^m} f(x, y)$. We use $diag(v)$ for a vector $v$ to denote the square matrix which has $v$ on the diagonal and zeros everywhere else.

We make the following assumptions about $f(x,y)$ throughout the paper.
\begin{assumption}\label{basic assumption}
	For the minimax problem \eqref{problem}, $f(x,y)$ satifies the following assumptions: 
	\begin{enumerate}
		\item[(1)] $f(x, y)$ is twice differentiable, $\mu$-strongly concave with respect to $y$ and non-convex with respect to $x$.
		\item[(2)] Denote $z = (x, y)$. $f(z)$ is $\ell$-smooth, i.e., for any $z, z'$, it holds:
		\[
		\|\nabla f(z) - \nabla f(z') \|\le \ell \|z - z'\|.
		\]
		\item[(3)] The Jacobian matrices $\nabla_{xx}^2 f(x, y)$, $\nabla_{xy}^2 f(x, y)$, $\nabla_{yx}^2 f(x, y)$ and $ \nabla_{yy}^2 f(x, y)$ are $\rho$-Lipschitz continuous.
		\item[(4)] $P(x)$ is bounded below by a constant $P^*$ and has compact sub-level sets.
	\end{enumerate}
\end{assumption}

\begin{definition}
	Under Assumption \ref{basic assumption}, we define the condition number of $f(x,y)$ as $\kappa := \ell/\mu$.
\end{definition}

\begin{lemma}\quad[\cite{chencubic}, Proposition 1]\label{gradient of P}
	Suppose $f(x,y)$ satisfies Assumption \ref{basic assumption}, then $P(x)$ has
	$L_P:= (\kappa+1)\ell$-Lipschitz continuous gradients. Moreover, $y^*(x)$ is well-defined and $\kappa$-Lipschitz, and $\nabla P(x) = \nabla_x f(x, y^*(x))$.
\end{lemma}

\begin{definition}\label{approximate FSP}
	Suppose $f(x,y)$ satisfies Assumption \ref{basic assumption},
	then we call $x$ an $\eps$-first order stationary point(FSP) of $P(x)$ if $\norm{\nabla{P(x)}}\leq\eps$.
\end{definition}

\begin{definition}\label{approximate SSP}
	Suppose $f(x,y)$ satisfies Assumption \ref{basic assumption}, then we call $x$ is an$(\eps, \sqrt{\eps})$-second order stationary point(SSP) of $P(x)$ if $\norm{\nabla P(x)}\leq\eps$ and $\nabla^2 P(x)\succeq -\sqrt{\eps} \mathbf I$.
\end{definition}

\begin{lemma}\quad[\cite{chencubic}, Proposition 2]\label{hessian of P}
	Suppose $f(x,y)$ satisfies Assumption \ref{basic assumption}. Denote $H(x,y) = [\nabla_{xx}^2 f - \nabla_{xy}^2 f (\nabla_{yy}^2 f)^{-1}\nabla_{yx}^2 f](x,y)$, then $H(x,y)$ is a Lipschitz continuous mapping with Lipschitz constant $L_H = \rho (1 + \kappa)^2$. Additionally, the Hessian of $P(x)$ satisfies $\nabla^2 P(x)=H(x,y^*(x))$, and it is Lipschitz continuous with constant $H_{Lip} = \rho (1+\kappa)^3$.
\end{lemma}

\begin{lemma}\quad[\cite{nesterov2006cubic}, Lemma 1]\label{lem:2.3}
	If Assumption \ref{basic assumption} and Lemma \ref{hessian of P} hold, we have
	\begin{equation}\label{gradient of P first order expansion error}
	\norm{\nabla P(y) - \nabla P(x) - \nabla^2 P(x)(y-x)} \leq 
	\frac{H_{Lip}}{2} \norm{x-y}^2	
	\end{equation}
	and
	\begin{equation}\label{P second order expansion error}
	\left|
	P(y)-P(x)-\nabla P(x)^\top (y-x) - \frac{1}{2} (y-x)^\top \nabla^2 P(x) (y-x)
	\right| \leq \frac{H_{Lip}}{6} \norm{x-y}^3.
	\end{equation}
\end{lemma}

\begin{lemma}\quad[\cite{conn2000trust}, Corollary 7.2.2]\label{TR sub}
	For a twice differentiable function $F(x)$, the vector $s^*$ is an optimal solution of the trust region subproblem
	\begin{equation}
	\min\limits_{s \in \mathbb{R}^n} F(x) + \nabla F(x)^T s + \frac{1}{2} s^T \nabla^2 F(x) s, \quad \text{s.t. } \norm{s} \leq r
	\end{equation}
	if and only if $s^*$ is feasible and there is a scalar $\lambda \geq 0$ such that the following coniditions are satisfied:
	\begin{subequations}
		\begin{align}
		\nabla F(x) + (\nabla^2 F(x) + \lambda \mathbf{I}) s^* = 0, \label{first opt condition} \\ 
		(\nabla^2 F(x) + \lambda \mathbf{I}) \succeq 0, \label{second opt condition}\\
		\lambda (r - \norm{s^*}) = 0. \label{complementary condition}
		\end{align}
	\end{subequations}
\end{lemma}

\section{A Minimax Trust Region Algorithm}
In this section, we propose a Minimax Trust Region (MINIMAX-TR) algorithm for solving \eqref{problem}.
At each iteration, for a given $x_t$, MINIMAX-TR algorithm runs enough number of gradient ascent steps on $y$ to obtain $y_t$ which is a good estimate of $y^*(x_t)$. Consequently, $g_t = \nabla_x f(x_t, y_t) $ and $H_t = H(x_t, y_t)$ are good approximations of $\nabla P(x_t)=\nabla_x f(x_t, y^*(x_t))$ and $\nabla^2 P(x_t)= H(x_t, y^*(x_t))$ respectively. Then we use the trust region method to update $x_t$. MINIMAX-TR can be regarded as an inexact trust region method for solving $\min_{x\in\mathbb{R}^n} P(x)$. 
The proposed MINIMAX-TR is formally stated in Algorithm \ref{MINIMAX-TR}.

\begin{algorithm}[h!]
	\footnotesize
	\caption{\quad Minimax Trust Region (MINIMAX-TR)}
	\label{MINIMAX-TR}	
	
	\begin{algorithmic}[1]
			\REQUIRE $x_0, y_{-1}=\mathbf 0, \eps, r, \eta_y, T, \{N_t\}_{t=0}^T$; 
			\FOR{$t=0, \cdots T-1$}
				\STATE $\widetilde{y}_0 = y_{t-1}$;
				\FOR{$k=0, \cdots N_t-1$}
					\STATE $\widetilde{y}_{k+1} = \widetilde{y}_k + \eta_y \nabla_y f(x_t, \widetilde{y}_k)$;
				\ENDFOR
				\STATE $y_t = \widetilde{y}_{N_t}$;
				\STATE $g_t = \nabla_x f(x_t,y_t)$;
				\STATE $H_t = H(x_t,y_t)$;
				\STATE $s_t=\mathop{\arg\min}\limits_{\norm{s}\leq r} g_t ^\top s + \frac{1}{2} s^\top H_t s$;
				\STATE $x_{t+1} = x_t + s_t$;
				\IF{$\lambda_t \leq 2\sqrt{\eps/ H_{Lip}}$}
				\STATE Output $x_{t+1}$;
				\ENDIF
			\ENDFOR
	\end{algorithmic}
\end{algorithm}

The following lemma shows that we can obtain sufficient accurate gradient and Hessian estimators by running gradient ascent with enough number of iterations. 
\begin{lemma}\quad\label{gradient ascent iterations}
	For any given $\eps_1 > 0$ and $\eps_2 > 0$, if Assumption \ref{basic assumption} holds, and we choose
	\begin{equation}\label{N_t}
	\left\{
	\begin{split}
	N_0 &\geq \kappa \ln \frac{\norm{y_{-1}-y^*(x_0)}}{A}, \\
	N_t &\geq \kappa \ln \frac{A+\kappa\norm{s_{t-1}} }{A}.  \hspace{10mm} t \geq 1
	\end{split}
	\right.
	\end{equation}
	where $A = \min\{\eps_1/l, \eps_2/L_H\}$. Then $\norm{\nabla P(x_t)-g_t} \leq \eps_1 $ and
	$\norm{\nabla^2 P(x_t)-H_t} \leq \eps_2$.
\end{lemma}
\begin{proof}
	We first use induction to show that
	\begin{equation}\label{error y}
	\norm{y^*(x_t)-y_t} \leq A
	\end{equation}
	holds for any $t \geq 0$. Note that $y_t$ is obtained by applying $N_t$ gradint descent steps starting from $y_{t-1}$. Set $\eta_y = \frac{2}{l+\mu}$. Then, by Assumption \ref{basic assumption} and Theorem 2.1.15 in \cite{nesterov2018lectures}, we conclude that
	\begin{equation}\label{GD rate}
	\norm{y_t-y^*(x_t)} \leq (1-\kappa^{-1})^{N_t} \norm{y_{t-1}-y^*(x_t))}.
	\end{equation}
	Firstly, by setting $t=0$ in \eqref{GD rate}, we get
	\begin{equation}
	\begin{split}
	\nonumber
	\norm{y_0-y^*(x_0)} &\leq \norm{y_{-1}-y^*(x_0)}\exp(N_0\ln(1-\kappa^{-1}))  \\
	&\leq \norm{y_{-1}-y^*(x_0)}\exp(-ln\frac{\norm{y_{-1}-y^*(x_0)}}{A})=A	
	\end{split}
	\end{equation}
	where the second inequality is by \eqref{N_t} and the fact that $\ln(1- \kappa^{-1})\leq -\kappa^{-1}$ when $\kappa^{-1}\in(0,1)$. Hence \eqref{error y} holds for $t=0$. Suppose \eqref{error y} holds for $t = t' - 1$. Then, by the Cauchy-Swartz ineuality and the $\kappa$-Lipschitz continuous of $y^*(x)$,  \eqref{GD rate} implies that
	\begin{equation}
	\begin{split}
	\nonumber
	\norm{y_{t'} - y^*(x_{t'})} &\leq (1-\kappa^{-1})^{N_{t'}} (\norm{y_{t'-1}-y^*(x_{t'-1})}+\norm{y^*(x_{t'-1})-y^*(x_{t'})}) \\
	&\leq \exp(N_{t'}\ln(1-\kappa^{-1}))(A+\kappa\norm{s_{t'-1}}) \\
	&\leq \exp(-\ln(\frac{A+\kappa\norm{s_{t'-1}}}{A})) (A+\kappa\norm{s_{t'-1}}) = A,
	\end{split}
	\end{equation}    
	where the last inequality is by $\ln(1-\kappa^{-1})\le -\kappa^{-1}$ and \eqref{N_t}. By induction, we complete the proof of \eqref{error y}.
	
	By Assumption \ref{basic assumption}, Lemma \ref{gradient of P}, Lemma \ref{hessian of P} and \eqref{error y}, it can be easily verified that
	\begin{align}
	\norm{\nabla P(x_t) - g_t} &= \norm{\nabla_x f(x_t,y^*(x_t)) - \nabla_x f(x_t,y_t)} \leq 
	l\norm{y_t-y^*(x_t)} \leq \eps_1,\nonumber\\
	\norm{\nabla^2 P(x_t) - H_t} &= \norm{H(x_t,y^*(x_t))-H(x_t,y_t)} \leq L_H \norm{y_t-y^*(x_t)} \leq \eps_2,\nonumber
	\end{align}
	which complete the proof.
\end{proof}  

Using Lemma \ref{gradient ascent iterations}, we show that MINIMAX-TR can find $\mathcal{O}(\eps,\sqrt{\eps})$-SSP in less than $T = \mathcal{O}(\eps^{-1.5})$ iterations.
\begin{theorem}\quad\label{Minimax Trust Region convergence rate}
	Suppose Assumption \ref{basic assumption} holds. By setting $\eps_1 = \eps/12, \eps_2 = \sqrt{\eps H_{Lip}}/6$, $r=\sqrt{\eps/H_{Lip}}$ and choosing $N_t$ in \eqref{N_t}, Algorithm \ref{MINIMAX-TR} finds an $\mathcal{O}(\eps,\sqrt{\eps})$-SSP in less than  $T =6\sqrt{H_{Lip}}(P(x_0)-P^*)/\epsilon^{1.5}$ iterations.
\end{theorem}
\begin{proof} 
	By Lemma \ref{lem:2.3} and the Cauchy-Swartz inequality,
	\begin{equation}\label{prove-descent}
	\begin{split}
	&P(x_{t+1}) - P(x_t) \\
	\leq &\nabla P(x_t)^\top s_t + \frac{1}{2} s_t^\top \nabla^2 P(x_t) s_t + \frac{H_{Lip}}{6} \norm{s_t}^3 \\
	=&g_t^\top s_t + \frac{1}{2} s_t^\top H_t s_t + (\nabla P(x_t) - g_t)^\top s_t + \frac{1}{2} s_t^\top (\nabla^2 P(x_t) - H_t) s_t + \frac{H_{Lip}}{6} \norm{s_t}^3 \\
	\leq&g_t^\top s_t + \frac{1}{2} s_t^\top H_t s_t + \norm{\nabla P(x_t) - g_t} \norm{s_t} + \frac{1}{2} \norm{\nabla^2 P(x_t) - H_t} \norm{s_t}^2 + \frac{H_{Lip}}{6} \norm{s_t}^3.
	\end{split}
	\end{equation}
	By Lemma \ref{gradient ascent iterations} and  $\norm{s}\leq r= \sqrt{\eps/H_{Lip}}$, we obtain
	\begin{equation}\label{prove descent eq1}
	\norm{\nabla P(x_t) - g_t} \norm{s_t} + \frac{1}{2} \norm{\nabla^2 P(x_t) - H_t} \norm{s_t}^2 \leq \frac{1}{6} \frac{\eps^{1.5}}{\sqrt{H_{Lip}}}.
	\end{equation}
	By Lemma \ref{TR sub}, we have that there exists a dual variable $\lambda_t \geq 0$ such that
	\begin{align}
	g_t + H_t s_t + \frac{\lambda_t H_{Lip}}{2} s_t& = 0, \label{opt first} \\
	H_t + \frac{\lambda_t H_{Lip}}{2} \mathbf{I}
	&\succcurlyeq 0, 
	\label{opt second} \\
	\lambda_t(\norm{s_t}-r)&=0. \label{opt complementary}
	\end{align}	
	By \eqref{opt first} and \eqref{opt second}, we have
	\begin{equation}\label{prove descent eq2}
	g_t^\top s_t = - (H_t s_t + \frac{\lambda_t H_{Lip}}{2}s_t)^\top s_t \leq 0.
	\end{equation}
	Moreover, by Algorithm \ref{MINIMAX-TR} if we have $\lambda_t \geq 2\sqrt{\epsilon/H_{Lip}} >0$. Then, the complementary property \eqref{opt complementary} implies that $\norm{s_t}=\sqrt{\eps/H_{Lip}}$.
	Plugging \eqref{prove descent eq1}, \eqref{opt first} and \eqref{prove descent eq2} into \eqref{prove-descent}, and using $\norm{s_t}=\sqrt{\eps/H_{Lip}}$, we obtain
	\begin{align*}
	P(x_{t+1}) - P(x_t) &\leq \frac{1}{2} g_t^\top s_t + \frac{1}{2} (H_t s_t + g_t)^\top s_t + \frac{H_{Lip}}{6}\norm{s_t}^3 + \frac{1}{6}\frac{\eps^{1.5}}{\sqrt{H_{Lip}}} \\
	&\leq - \frac{H_{Lip} \lambda_t}{4} \frac{\eps}{H_{Lip}} + \frac{1}{3} \frac{\eps^{1.5}}{\sqrt{H_{Lip}}}.
	\end{align*}
	Therefore, if we have $\lambda_t \geq 2\sqrt{\epsilon/H_{Lip}}$, then
	\begin{equation}\label{P-decrease}
	P(x_{t+1}) - P(x_t) \leq -\frac{1}{6} \frac{\eps^{1.5}}{\sqrt{H_{Lip}}}.
	\end{equation}
	Denote $\bar{T}=\min\{t: \lambda_t \leq 2\sqrt{\eps/ H_{Lip}} \}$.
	By summing \eqref{P-decrease} from $0$ to $\bar{T}$ and using the lower bounded property of $P(x)$ in Assumption \ref{basic assumption},  it can be easily verified that $\bar{T}\leq 6\sqrt{H_{Lip}}(P(x_0)-P^*)/\epsilon^{1.5}$.
	If $\lambda_t \leq 2\sqrt{\eps/ H_{Lip}}$, then we prove that $x_{t+1}$ is already an $\mathcal{O}(\eps,\sqrt{\eps})$-SSP as follows. 
	
	On one hand, by Lemma \ref{lem:2.3}, the Cauchy-Swartz inequality and \eqref{opt first}
	\begin{equation}
	\begin{split}
	&\norm{\nabla P(x_{t+1})} \\
	= &\norm{
		\nabla P(x_{t+1}) - \nabla P(x_t) - \nabla^2 P(x_t) s_t + \nabla P(x_t) + \nabla^2 P(x_t) s_t
	} \\
	\leq & \frac{H_{Lip}}{2}\norm{s_t}^2 + \norm{
		\nabla P(x_{t}) - g_t + (\nabla^2 P(x_t) - H_t) s_t
	} + \norm{g_t + H_t s_t} \\
	\leq & \frac{H_{Lip}}{2}\norm{s_t}^2 + \norm{\nabla P(x_{t}) - g_t} + \norm{\nabla^2 P(x_t) - H_t} \norm{s_t} + \norm{\frac{\lambda_t H_{Lip}}{2} s_t} \\
	\leq & \frac{7}{4} \eps,
	\end{split}
	\end{equation}
	where the last inequality is by Lemma \ref{gradient ascent iterations}, the choices of $\eps_1 = \eps/12, \eps_2 = \sqrt{\eps H_{Lip}}/6$, and the fact that $\|s_t\|\leq r=\sqrt{\eps/H_{Lip}}$. 
	On the other hand, by Lemma \ref{hessian of P}, Lemma \ref{gradient ascent iterations} and \eqref{opt second}, we have
	\begin{equation}
	\begin{split}
	&\nabla^2 P(x_{t+1}) \\
	\succcurlyeq & \nabla^2 P(x_t) - H_{Lip} \norm{s_t} \mathbf I \\
	= &H_t - (H_t - \nabla^2 P(x_t)) - H_{Lip} \norm{s_t} \mathbf I \\
	\succcurlyeq & -\frac{\lambda_t H_{Lip}}{2} \mathbf I - (H_t - \nabla^2 P(x_t)) - H_{Lip} \norm{s_t} \mathbf I \\
	\succcurlyeq & - \frac{13}{6} \sqrt{H_{Lip} \eps} \mathbf I,
	\end{split}
	\end{equation}
	which completes the proof.
\end{proof}

\section{A Minimax Trust Region Algorithm with Contractions and Expansions}
In this section, we propose a minimax trust region algorithm with contractions and expansions (MINIMAX-TRACE) for solving \eqref{problem}. The proposed algorithm is inspried by TRACE \cite{TRACE}, which is a trust region variant for solving minimization optimization problem. Firstly, we briefly introduce the TRACE method proposed by Curtis et al. \cite{TRACE} for solving $\min_{x\in\mathbb{R}^n} P(x)$. At the $t$th iteration, it solves the following trust region subproblem to obtain a trial step $s_t$:
\begin{equation}
s_t = \mathop{\arg\min}\limits_{\norm{s} \leq \delta_t} p_t(s),\label{eq:4.1} 
\end{equation}
where $p_t(s) = P(x_t) + \nabla P(x_t)^T s + \frac{1}{2}s^T \nabla^2 P(x_t) s$. 
Then, the ratio of actual-to-predicted reduction is defined as follows, 
\begin{equation}
\rho_t:=\frac{P(x_t) - P(x_t + s_t)}{\norm{s_t}^3}.\label{rho} 
\end{equation}
For some prescribed $\eta \in (0,1)$, the trial step $s_t$ can only be accepted if $\rho_t \geq \eta$.
Otherwise, TRACE employs a CONTRACT subroutine to adjust the trust region radius. More detailedly, it compares the radius that would be obtained via a traditional updating scheme to the norm of the trial step obtained from the following subproblem
\begin{equation}
\min\limits_{s\in\mathbb{R}^n} P(x_t) + \nabla P(x_t)^T s + \frac{1}{2} s^T (\nabla^2 P(x_t) + \lambda \mathbf I) s \label{eq:4.3} 
\end{equation}
for a carefully chosen $\lambda>0$. If the norm of the step resulting from this procedure falls into a suitable range, then it is employed as the trust region radius in the subsequent iteration, as opposed to updating the radius in a more traditional manner. TRACE also proposes a novel expansion procedure which expand the trust region to avoid steps that too small in norm.

The proposed MINIMAX-TRACE algorithm can be regarded as an inexact version of TRACE for solving \eqref{problem} with $P(x) := \max_{y\in \mathcal{Y}}f(x,y)$. Similar to MINIMAX-TR, at each iteration, MINIMAX-TRACE runs enough number of gradient ascent steps on $y$ to obtain $g_t = \nabla_x f(x_t, y_t) $ and $H_t = H(x_t, y_t)$ which are good approximations of $\nabla P(x_t)=\nabla_x f(x_t, y^*(x_t))$ and $\nabla^2 P(x_t)= H(x_t, y^*(x_t))$ respectively. Then it solves the following trust region sub-problem
\begin{equation}
Q_t : \min\limits_{\norm{s} \leq \delta_t} q_t(s),
\end{equation}
where $q_t(s) = P(x_t) + g_t^T s + \frac{1}{2}s^T H_t s$. Note that $q_t(s)$ is an approximation of $p_t(s)$ in \eqref{eq:4.1} since $g_t$ and $H_t$ are approximations of $\nabla P(x_t)$ and $\nabla^2 P(x_t)$.
To decide whether to accept or reject a trial step, MINIMAX-TRACE computes the ratio $\rho_t$ defined as in \eqref{rho}. For some prescribed $\eta \in (0,1)$, the trial step $s_t$ can only be accepted if $\rho_t \geq \eta$. Otherwise, MINIMAX-TRACE employs a similar CONTRACT subroutine to adjust the trust region radius. The main difference is that it compares the radius that would be obtained via a traditional updating scheme to the norm of the trial step obtained from the following subproblem
\begin{equation}
Q_t(\lambda) : \min\limits_{s\in\mathbb{R}^n} P(x_t) + g_t^T s + \frac{1}{2} s^T (H_t + \lambda \mathbf I) s 
\end{equation}
for a carefully chosen $\lambda$ to obtain the new trust region radius, which can be regarded as an inexact version of \eqref{eq:4.3}. The proposed MINIMAX-TRACE is formally stated in Algorithm \ref{MINIMAX-TRACE}.

\begin{algorithm}[h!]
	\footnotesize
	\caption{Minimax Trust Region Algorithm with Contractions and Expansions (MINIMAX-TRACE)}
	\label{MINIMAX-TRACE}	
	\begin{algorithmic}[1]
		\REQUIRE $y_{-1}=\mathbf 0, \eta_y, \eta \in (0,1), 0 < \gamma_{C} < 1 < \gamma_E, \gamma_{\lambda} > 1, 0 < \underline{\sigma} \leq \overline{\sigma}, \delta_0, \Delta_0 > 0$ with $\delta_0 \leq \Delta_0$, and $\sigma_0>0$ with $\sigma_0 \geq \underline{\sigma}$;
		\FOR{$t=0, \cdots T-1$}
			\STATE $\widetilde{y}_0 = y_{t-1}$; 
			\FOR{$k=0, \cdots N_t-1$}
				\STATE $\widetilde{y}_{k+1} = \widetilde{y}_k + \eta_y \nabla_y f(x_t, \widetilde{y}_k)$;
			\ENDFOR
			\STATE $y_t = \widetilde{y}_{N_t}$;
			\STATE $g_t = \nabla_x f(x_t,y_t)$;
			\STATE $H_t = H(x_t,y_t)$;
			\IF{$t = 0$} 
				\STATE compute $(s_0,\lambda_0) $ by solving $Q_0$, then set $\rho_0$ as in \eqref{rho};
			\ENDIF
			\STATE $x_{t+1}, \delta_{t+1}, \Delta_{t+1}, \sigma_{t+1} =  \text{TRACE}(x_t,s_t,\lambda_t,\delta_t,\Delta_t,\sigma_t,\rho_t)$;
			
			\STATE compute $(s_{t+1},\lambda_{t+1}) $ by solving $Q_{t+1}$, then set $\rho_{t+1}$ as in \eqref{rho};
			\IF{$\rho_t < \eta$}
				\STATE $\sigma_{t+1} = \max \{ \sigma_t,\lambda_{t+1}/\norm{s_{t+1}}\}$;
			\ENDIF
		\ENDFOR
	\end{algorithmic}
\end{algorithm}

\begin{algorithm}[h!]
	\footnotesize
	\caption{TRACE}	
	\label{TRACE}
	\begin{algorithmic}[1]
		\REQUIRE $x_t,s_t,\lambda_t,\delta_t,\Delta_t,\sigma_t,\rho_t$; 
		\IF{$\rho_t \geq \eta$ and either $\lambda_t \leq \sigma_t \left\Vert s_t \right\Vert$ or $\norm {s_t}$=$\Delta_t$} 
		\STATE $x_{t+1} = x_t + s_t$; 
		\STATE $\Delta_{t+1} = \max \{ \Delta_t,\gamma_E \norm {s_t} \} $;
		\STATE $\delta_{t+1} = \min\{\Delta_{t+1},\max\{\delta_t,\gamma_E \norm {s_t} \}\}$;
		\STATE $\sigma_{t+1} = \max\{\sigma_t,\lambda_t/ \norm {s_t}\}$;
		\ELSIF{$\rho_t < \eta$} 
			\STATE $x_{t+1} = x_t$;
			\STATE $\Delta_{t+1}=\Delta_t$;
			\STATE $\delta_{t+1}=\text{CONTRACT}(x_t,\delta_t,\sigma_t,s_t,\lambda_t)$;
		\ELSIF{$\rho_t \geq \eta,\lambda_t > \sigma_t \norm{s_t}, \text{ and } \norm{s_t} < \Delta_t$}
			\STATE $x_{t+1} = x_t$;
			\STATE $\Delta_{t+1} = \Delta_t$; 
			\STATE $\delta_{t+1} = \min \{ \Delta_{t+1},\lambda_t/\sigma_t \}$;
			\STATE $\sigma_{t+1} = \sigma_t$;
		\ENDIF
		\RETURN $x_{t+1}, \delta_{t+1}, \Delta_{t+1}, \sigma_{t+1}$;
	\end{algorithmic}		
\end{algorithm}

\begin{algorithm}[h!]
	\footnotesize		
	\caption{CONTRACT}
	\label{CONTRACT}
	\begin{algorithmic}[1]
		\REQUIRE $x_t,\delta_t,\sigma_t,s_t,\lambda_t$;
		\IF{$\lambda_t<\underline{\sigma}\norm {s_t}$}
			\STATE $\hat{\lambda}=\lambda_t + (\underline{\sigma}\norm {g_t})^{1/2}$;
			\STATE $\lambda = \hat{\lambda}$;
			\STATE set $s^1$ as the solution of $Q_t(\lambda)$;
			\IF{$\lambda/\norm{s^1} \leq \overline{\sigma}$}
			\STATE \textbf{return} $\delta_{t+1} = \norm{s^1}$;
			\ELSE
				\STATE compute $\lambda \in (\lambda_t,\hat{\lambda})$ so the solution $s^2$ of $Q_t(\lambda)$ yields $\underline{\sigma} \leq \lambda/\norm{s^2} \leq \overline{\sigma}$;
				\RETURN $\delta_{t+1} = \norm{s^2}$;
			\ENDIF
		\ELSE 
			\STATE $\lambda = \gamma_{\lambda}\lambda_t$; 
			\STATE set $s^3$ as the solution of $Q_t(\lambda)$;
			\STATE $\delta_{t+1}=\max \{\norm{s^3}, \gamma_C \norm {s_t}\}$.
		    \STATE \textbf{return} $\delta_{t+1}$.
		\ENDIF
	\end{algorithmic}
\end{algorithm}

For convenience, we distinguish between different types of iterations by partitioning the set of iteration numbers into the sets of accepted ($\mathcal{A}$), contraction ($\mathcal{C}$), and expansion ($\mathcal{E}$) steps as follows,
\begin{align*}
\mathcal{A} &:= \{
t\in \mathbb{N}: \rho_t \geq \eta \text{ and either } \lambda_t \leq \sigma_t \norm{s_t} \text{ or } \norm{s_t} = \Delta_t
\}, \\
\mathcal{C} &:= \{
t\in \mathbb{N}: \rho_t < \eta
\}, \\
\mathcal{E} &:= \{
t\in \mathbb{N}: t \notin \mathcal{A} \cup \mathcal{C}
\}.
\end{align*}
We also partition the set of accepted steps into two disjoint subsets $\mathcal{A}_\Delta := \{ t \in \mathcal{A} : \norm {s_t} = \Delta_t \}$ and $\mathcal{A}_\sigma := \{t \in \mathcal{A} : t \notin \mathcal{A}_{\Delta}\}$.
%
Firstly, we make the following assumption for Algorithm \ref{MINIMAX-TRACE}.
\begin{assumption}\label{bound assumption}
	$\norm{\nabla P(x)}$ is upper bounded by a constant $M_1>0$.
	We assume the gradient estimator sequence $\{ g_t \}$ has $g_t \neq 0$ for all $t \in \mathbb{N}$. Furthermore, suppose there exists sufficiently large $N_t$ and a constant $M_2 > 0$, such that $\norm{y_t - y^*(x_t)} \leq \min\{ C_1\norm{s_t}^2/l, M_2/l, C_2 \norm{s_t}/L_H, M_2/L_H \} $ for all $t \in \mathbb{N}$.
\end{assumption}
Although we do not prove a theoretical bound for $N_t$ that in Assumption \ref{bound assumption}. We know that if $\norm{s_t}$ is in the order of $\mathcal{O} (\eps^\alpha)$, many algorithms including Nestrov's accelerated gradient algorithm can find a $y_t$ that satisfies Assumption \ref{bound assumption} after $N_t = \mathcal{O} (\alpha\log \frac{1}{\eps})$ iterations under strong concavity. Compared to the order of the number of iterations $\mathcal{O}(\eps^{-1.5})$ that will be proved latter in this section, it can be ignored. In practice, since $\|s_t\|>0$ always holds, it can be easily satisfied.   

\begin{lemma}\quad\label{bound lemma}
	Suppose Assumptions \ref{basic assumption} and \ref{bound assumption} hold. Then, we have that $\norm{\nabla P(x_t)-g_t} \leq \min\{C_1 \norm{s_t}^2,M_2\} $ and
	$\norm{\nabla^2 P(x_t)-H_t} \leq \min\{C_2 \norm{s_t},M_2\}$ for all $t \in \mathbb{N}$.
\end{lemma}
\begin{proof}
	By Lemmas \ref{gradient of P} and \ref{hessian of P}, Assumptions \ref{basic assumption} and \ref{bound assumption}, we get 
	\begin{align}
	\norm{\nabla P(x_t) - g_t} &= \norm{\nabla_x f(x_t,y^*(x_t)) - \nabla_x f(x_t,y_t)} \nonumber\\
	&\leq 
	l\norm{y_t-y^*(x_t)} \leq \min\{C_1 \norm{s_t}^2,M_2\},\nonumber\\
	\norm{\nabla^2 P(x_t) - H_t} &= \norm{H(x_t,y^*(x_t))-H(x_t,y_t)} \nonumber\\
	& \leq L_H \norm{y_t-y^*(x_t)} \leq \min\{C_2 \norm{s_t},M_2\},\nonumber
	\end{align}
	which complete the proof.
\end{proof}

By Lemma \ref{bound lemma}, we can prove the following lemma which shows that $\{ \norm{g_t}\}$ and $\{ \norm{H_t}\}$ are upper bounded.
\begin{lemma}\quad
	If Assumptions \ref{basic assumption} and \ref{bound assumption} hold, then we have $\norm{g_t} \leq g_{max}:= M_1 + M_2$ and $\norm{H_t} \leq H_{max}:= M_2 + L_P$ for all $t \in \mathbb{N}$.
\end{lemma}
\begin{proof}
	By Assumption \ref{bound assumption}, the Cauchy-Swartz inequality, and Lemma \ref{bound lemma}, we get that 
	\begin{align*}
	\norm{g_t} &= \norm{g_t - \nabla P(x_t)} + \norm{\nabla P(x_t)} \leq M_2 + M_1,\\
	\norm{H_t} &= \norm{H_t - \nabla^2 P(x_t)} + \norm{\nabla^2 P(x_t)} \leq M_2 + L_P,
	\end{align*}
	which completes the proof.
\end{proof}

By Lemma 3.2 and Lemma 3.3 in \cite{TRACE}, we immediately obtain a general lower bound of $\norm{s_t}$ and the estimate of decrease in the model function shown in the following two lemmas.
\begin{lemma}\quad\label{s lower bound}
	For any $t \in \mathbb{N}$, the trial step $s_t$ satisfies \begin{equation}
	\norm{s_t} \geq \min \left\{
	\delta_t,\frac{\norm{g_t}}{\norm{H_t}}
	\right\} > 0.
	\end{equation}
\end{lemma}
\begin{lemma}\quad\label{gap between P(x_t) and q_t(s_t)}
	For any $t \in \mathbb{N}$, the trial step $s_t$ and dual variable $\lambda_t$ satisfy 
	\begin{equation}
	P(x_t) - q_t(s_t) = \frac{1}{2} s_t^T(H_t + \lambda_t I)s_t + \frac{1}{2} \lambda_t \norm{s_t}^2.
	\end{equation}
	In addition, for any $t \in \mathbb{N}$, the trial steps $s_t$ satisfies
	\begin{equation}
	P(x_t) - q_t(s_t) \geq \frac{1}{2} \norm{g_t} \min \left\{ \delta_t, \frac{\norm{g_t}}{\norm{H_t}}
	\right\}.
	\end{equation}
\end{lemma}

%
%

The next lemma shows that the trust region radius is reduced when the CONTRACT subroutine is called, otherwise the trust region radius is increased. Moreover, the trust region radius $\delta_t$ is upper bounded by a nondecreasing sequence $\{ \Delta_t\}$.
\begin{lemma}(Lemmas 3.4-3.6 in \cite{TRACE}) \label{radius relation}
	For any $t \in \mathbb{N}$, if $t \in \mathcal{C}$, then $\delta_{t+1} < \delta_t$ and $\lambda_{t+1} \geq \lambda_t$, otherwise, $\delta_{t+1} \geq \delta_t$. In addition, there holds $\delta_t \leq \Delta_t \leq \Delta_{t+1}$ for any $t \in \mathbb{N}$.
\end{lemma}


Next, we show that if the dual variable for the trust region constraint $\lambda_t$ is sufficiently large, the trial step yields sufficient reduction in the objective.
\begin{lemma}\quad\label{lambda success step}
	For any $t \in \mathbb{N}$, if the trial step $s_t$ and dual variable $\lambda_t$ satisfy 
	\begin{equation}\label{lambda assumption1}
	\lambda_t \geq 2 L_P + H_{max} + 2 (\eta + C_1)    
	\norm{s_t},
	\end{equation}
	then $\norm{s_t} = \delta_t$ and $\rho_t \geq \eta$.	
\end{lemma}
\begin{proof}
	By \eqref{lambda assumption1} we have $\lambda_t > 0$, then $\norm{s_t} = \delta_t$ which follows directly from \eqref{complementary condition} in Lemma \ref{TR sub}. Since $P(x)$ is a smooth function, there exists a point $\overline{x}_t \in [x_t,x_t + s_t]$ such that 
	\begin{equation}\label{gap between q(s_t) and P_t(x_t+s_t)}
	\begin{split}
	q_t(s_t) - P(x_t + s_t) &= P(x_t) + g_t^T s_t + \frac{1}{2} s_t^T H_t s_t - P(x_t + s_t)\\
	&=(g_t - \nabla P(\overline{x}_t))^T s_t + \frac{1}{2} s_t^T H_t s_t \\
	&\geq - \norm{g_t - \nabla P(\overline{x}_t))} \norm{s_t} - \frac{1}{2} \norm{H_t} \norm{s_t}^2.
	\end{split}
	\end{equation}
	From Lemma \ref{bound lemma}, Lemma \ref{gap between P(x_t) and q_t(s_t)}, \eqref{gap between q(s_t) and P_t(x_t+s_t)} and \eqref{lambda assumption1}, we have
	\begin{align}
	& P(x_t) - P(x_t + s_t) \nonumber\\
	= & P(x_t) - q_t(s_t) + q_t(s_t) - P(x_t + s_t) \nonumber\\
	\geq & \frac{1}{2} \lambda_t \norm{s_t}^2 - \norm{g_t - \nabla P(x_t)} \norm{s_t} - \norm{\nabla P(x_t) - \nabla P(\overline{x}_t)} \norm{s_t} - \frac{1}{2} \norm{H_t} \norm{s_t}^2\nonumber\\
	\geq & \frac{1}{2} \lambda_t \norm{s_t}^2 - C_1 \norm{s_t}^3 - L_P \norm{s_t}^2 - \frac{1}{2} H_{max} \norm{s_t}^2 \nonumber\\
	\geq & \eta \norm{s_t}^3,\label{4:11}
	\end{align}
	which implies that $\rho_t \geq \eta$. The proof is then completed.
\end{proof}


%

The next lemma shows that the set $\mathcal{A}$ has infinite cardinality and the set $\mathcal{A}_{\Delta}$ has finite cardinality.
\begin{lemma}\quad[\cite{TRACE}, Lemma 3.10-3.11]\label{setA}
	The set $\mathcal{A}$ has infinite cardinality. There exists a scalar constant $\Delta \in \mathbb{R}_{++}$ such that $\Delta_t = \Delta$ for all sufficiently large $t \in \mathbb{N}$, the set $\mathcal{A}_{\Delta}$ has finite cardinality, and there exists a scalar constant $s_{max} \in \mathbb{R}_{++}$, such that $\norm{s_t} \leq s_{max}$ for all $t \in \mathbb{N}$.
\end{lemma}

Next result ensures a lower bound on the trust region radius if the gradient sequence is asymptotically bounded away from zero.
\begin{lemma}\quad\label{radius lower bound}
	If there exists a scalar constant $g_{min} > 0$ such that 
	\begin{equation}
	\norm{g_t} \geq g_{min}  \text{ for all } t \in \mathbb{N},
	\end{equation}
	then there exists a scalar constant $\delta_{min} > 0$ such that $\delta_t \geq \delta_{min}$ for all $t \in \mathbb{N}$.
\end{lemma}
\begin{proof}
	If $|\mathcal{C}| < \infty$, the result holds by Lemma \ref{radius relation}. Then, we assume that $|\mathcal{C}| = \infty$.
	
	Firstly, we prove that there exits $\delta_{thresh} > 0$ such that, if $t \in \mathcal{C}$, then $\delta_t \geq \delta_{thresh}$. Indeed, similar to the proof of \eqref{4:11} in Lemma \ref{lambda success step} and by Lemma \ref{gap between P(x_t) and q_t(s_t)}, we have for all $t \in \mathbb{N}$
	that
	\begin{equation}\label{P descent}
	\begin{split}
	P(x_t) - P(x_t + s_t) &  = P(x_t) - q_t(s_t) + q_t(s_t) - P(x_t + s_t) \\
	&\geq \frac{1}{2} \norm{g_t} \min \left\{ \delta_t, \frac{\norm{g_t}}{\norm{H_t}} \right\} - C_1 \norm{s_t}^3 - L_P \norm{s_t}^2 - \frac{1}{2} H_{max} \norm{s_t}^2.
	\end{split} 
	\end{equation}
 Then, if $\delta_t \leq \delta_{thresh}= \min \{ \frac{g_{min}}{H_{max}}, \frac{L_p + \frac{1}{2}H_{max} + \sqrt{(L_p + \frac{1}{2}H_{max})^2 + 2g_{min}(C_1+\eta)}}{g_{min}} \}$ we have 
	\[
		P(x_t) - P(x_t + s_t) \geq \frac{1}{2} g_{min} \delta_t - C_1 \delta_t^3 - (L_P + \frac{1}{2} H_{max}) \delta_t^2 \geq \eta \delta_t^3 \geq \eta \norm{s_t}^3,
	\]
	which implies that $\rho_t \geq \eta$ that conflicts $t \in \mathcal{C}$. Hence, $\delta_t\geq \delta_{thresh}$.
	
	Suppose that $t \in \mathcal{C}$, we consider the update for $\delta_{t+1}$ in Step 9 of Algorithm \ref{TRACE}, i.e., Algorithm \ref{CONTRACT}. If Step 6 of Algorithm \ref{CONTRACT} is reached, then it follows that
	\[
	\delta_{t+1} = \norm{s^1} \geq \frac{\lambda}{\overline{\sigma}} = \frac{\lambda_t + (\underline{\sigma} \norm{g_t})^{1/2}}{\overline{\sigma}} \geq \frac{(\underline{\sigma} g_{min})^{1/2}}{\overline{\sigma}}.
	\]
	Otherwise, if Step 15 is reached, then it follows that
	\[
	\delta_{t+1} \geq \gamma_C \norm{s_t} = \gamma_C \delta_t \geq \gamma_C \delta_{thresh}.
	\]
    where the equality holds since that $\lambda_t > 0$ and \eqref{complementary condition} implies that $\norm{s_t} = \delta_t$. It remains to consider the value obtained in Step 9 is reached. 
	Since $\lambda \in (\lambda_t,\hat{\lambda})$, we have that $\norm{s^2} \geq \norm{s^1}$. Now let $H_t = V_t \Xi_t V_t^T$ where $V_t$ is an orthonormal matrix of eigenvectors and $\Xi_t = diag(\xi_{t,1},\dots,\xi_{t,n})$ with $\xi_{t,1}\leq \dots \leq \xi_{t,n}$ is a diagonal matrix of eigenvalues of $H_t$. Since $\hat{\lambda} > \lambda_t$, the matrix $H_t + \hat{\lambda} I$ is invertible and 
	\[
	\norm{s^1}^2 = \norm{V_t(\Xi_t + \hat{\lambda} I)^{-1} V_t^T g_t}^2 = g_t^T V_t (\Xi_t + \hat{\lambda} I)^{-2} V_t^T g_t.
	\]
	The orthonormality of $V_t$, Step 2, and Lemma \ref{setA} imply that
	\begin{equation}\nonumber
	\begin{split}
	\frac{\norm{s^1}^2}{\norm{g_t}^2} &= \frac{g_t^T V_t (\Xi_t + \hat{\lambda} I)^{-2} V_t^T g_t}{\norm{V_t^T g_t}^2} \\
	&\geq (\xi_{t,n} + \lambda_t + (\underline{\sigma} \norm{g_t})^{1/2})^{-2} \\
	&\geq (\xi_{t,n} + \underline{\sigma} \norm{s_t} + (\underline{\sigma} \norm{g_t})^{1/2})^{-2} \\
	&\geq (H_{max} + \underline{\sigma}s_{max} + (\underline{\sigma} g_{max})^{1/2}))^{-2}.
	\end{split}
	\end{equation}
	Hence, there exists a constant $s_{min}$ such that, for all such $t \in \mathcal{C}$, we have $\delta_{t+1} = \norm{s^2} \geq \norm{s^1} \geq s_{min}$ by $\norm{g_t} \geq g_{min}$. Combining all the three cases in the above analysis, we have shown that, for all $t \in \mathcal{C}$, we have
	\[
	\delta_{t+1} \geq \min \left\{
	\frac{(\underline{\sigma} g_{min})^{1/2}}{\overline{\sigma}},
	\gamma_C \delta_{thresh},
	s_{min} 
	\right\} > 0.
	\]
	The proof is then completed by combining the above proved results that $\delta_t\geq \delta_{thresh}$ when $t\in \mathcal{C}$, $\delta_{t+1}$ is lower bounded for $t \in \mathcal{C}$, and by Lemma \ref{radius relation} which ensures that $\delta_{t+1} \geq \delta_t$ for all $t \in \mathcal{A} \cup \mathcal{E}$.
\end{proof}


Next we prove the main global convergence result.
\begin{theorem}\quad \label{gradient convergence theorem}
	Under Assumption \ref{basic assumption} and Assumption \ref{bound assumption}, it holds that
	\begin{equation}\label{gradient convergence}
	\lim \limits_{t \in \mathbb{N},t \rightarrow \infty} \norm{\nabla P(x_t)} = 0.
	\end{equation}
\end{theorem}
\begin{proof}
	By Lemma \ref{radius lower bound}, similar to the proof of Lemma 3.13 in \cite{TRACE}, we first have that
	\begin{equation} \label{approx gradient liminf}
		\mathop{\lim\inf}\limits_{t \in \mathbb{N},t \rightarrow \infty} \norm{g_t} = 0.
	\end{equation} 
	Next, similar to Theorem 3.14 in \cite{TRACE}, we prove that
	\begin{equation}\label{approx gradient convergence}
	\lim \limits_{t \in \mathbb{N},t \rightarrow \infty} \norm{g_t} = 0.
	\end{equation} 
Suppose the contrary that \eqref{approx gradient convergence} does not hold. Combined with Lemma \ref{setA}, it implies that there exists an infinite subsequence $\{t_i\} \subseteq \mathcal{A}$ (indexed over $i \in \mathbb{N}$) such that, for some $\eps > 0$ and all $i \in \mathbb{N}$, we have $\norm{g_{t_i}} \geq 2\eps > 0$. Additionally,  Lemma \ref{setA} and \eqref{approx gradient liminf} imply that there exists an infinite subsequence $\{l_i\} \subseteq \mathcal{A}$ (indexed over $i \in \mathbb{N}$) such that, for all $i \in \mathbb{N}$ and $t \in \mathcal{A}$ with $t_i \leq t < l_i$, we have 
	\begin{equation}\label{gradient diff}
	\norm{g_t} \geq \eps \text{ and } \norm{g_{l_i}} < \eps.
	\end{equation}
	We now restrict our attention to indices in the infinite index set
	\[
	\mathcal{K} := \{t \in \mathcal{A}: t_i \leq t < l_i \text{ for some } i \in \mathbb{N}\}.
	\]
	Observe from \eqref{gradient diff} that, for all $t \in \mathcal{K}$, we have $\norm{g_t} \geq \eps$. Hence, by Lemma \ref{s lower bound}, we have for all $t \in \mathcal{K} \subseteq \mathcal{A}$ that 
	\begin{equation}\label{f desc}
	P(x_t) - P(x_{t+1}) \geq \eta \norm{s_t}^3 \geq \eta 
	\left(
	\min \left\{
	\delta_t, \frac{\eps}{H_{max}}
	\right\}
	\right)^3.
	\end{equation}
	Since $\{P(x_t)\}$ is monotonically decreasing and bounded below, we know that $P(x_t) \rightarrow \underline{P}$ for some $\underline{P} \in \mathbb{R}$, which when combined with \eqref{f desc} shows that
	\begin{equation}
	\lim\limits_{t\in\mathcal{K},t\rightarrow\infty} \delta_t = 0.
	\end{equation}
	Using this fact and \eqref{P descent}, we have for all sufficiently large $t \in \mathcal{K}$ that
	\begin{equation}\nonumber
	\begin{split}
	P(x_t) - P(x_t + s_t) &\geq \frac{1}{2} \norm{g_t} \min \left\{ \delta_t, \frac{\norm{g_t}}{\norm{H_t}} \right\} - C_1 \norm{s_t}^3 - (L_P + \frac{1}{2} H_{max}) \norm{s_t}^2 \\
	& \geq \frac{1}{2} \eps \min \left\{ \delta_t, \frac{\eps}{H_{max}} \right\} - C_1 \delta_t^3 - (L_P + \frac{1}{2} H_{max}) \delta_t^2 \\
	& \geq \frac{1}{2} \eps \delta_t - C_1 \delta_t^3 - (L_P + \frac{1}{2} H_{max}) \delta_t^2 \\
	& \geq \frac{\eps \delta_t}{4}.
	\end{split}
	\end{equation}
	Consequently, for all sufficiently large $i \in \mathbb{N}$, we have
	\begin{equation}\nonumber
	\begin{split}
	\norm{x_{t_i} - x_{l_i}} &\leq \sum_{t \in \mathcal{K},t=t_i}^{l_i - 1} \norm{x_t - x_{t+1}} \\
	&\leq \sum_{t \in \mathcal{K},t=t_i}^{l_i - 1} \delta_t 
	\leq \sum_{t \in \mathcal{K},t=t_i}^{l_i - 1} \frac{4}{\eps} (P(x_t) - P(x_{t+1})) 
	= \frac{4}{\eps} (P_{t_i} - P_{l_i}).
	\end{split}
	\end{equation}
	Since ${P_{t_i} - P_{l_i}} \rightarrow 0$, this implies that ${\norm{x_{t_i} - x_{l_i}}} \rightarrow 0$.
	
	For all $t \in \mathcal{A}$ we have $\rho_t \geq \eta$, which implies that $P(x_t) - P(x_{t+1}) \geq \eta \norm{s_t}^3$. Since $P(x)$ is bounded below on $\mathbb{R}^n$ and Lemma \ref{setA} ensures that $|\mathcal{A}| = \infty$, this implies that $\{s_t\}_{t\in \mathcal{A}} \rightarrow 0$.
	
	Using the fact that $\{t_i\}, \{l_i\} \subseteq\mathcal{A}$ and Lemma \ref{bound lemma}, we have
	\begin{equation}\nonumber
	\begin{split}
	{\norm{g_{t_i} - g_{l_i}}} &\leq \norm{g_{t_i} - \nabla P(x_{t_i})} + 
	\norm{\nabla P(x_{t_i}) - \nabla P(x_{l_i})} +
	\norm{\nabla P(x_{l_i}) - g_{l_i}} \\
	& \leq C_1 \norm{s_{t_i}}^2 + L_P \norm{x_{t_i} - x_{l_i}} + C_1 \norm{s_{l_i}}^2 \rightarrow 0.
	\end{split}
	\end{equation}
	Howerver, this is a contradiction since, for any $i \in \mathbb{N}$, we have $\norm{g_{t_i} - g_{l_i}} \geq \eps$ by the definitions of $\{t_i\}$ and $\{l_i\}$. Overall, we conclude that our initial supposition must be false, implying that \eqref{approx gradient convergence} holds.
	
	For all $t\in \mathcal{A}$, we have
	$\norm{\nabla P(x_t)} \leq \norm{\nabla P(x_t) - g_t} + \norm{g_t} \leq C_1 \norm{s_t}^2 + \norm{g_t} \rightarrow 0$, i.e., $\lim\limits_{t\in \mathcal{A}, t\rightarrow \infty} \norm{\nabla P(x_t)} = 0$. Then we prove that \eqref{gradient convergence} holds. Suppose \eqref{gradient convergence} does not hold. Then there exists an $\eps > 0$, such that for every $N > 0$, there exists a $t > N$ such that $\norm{\nabla P(x_t)} \geq \eps$. If $t \notin \mathcal{A}$, there exists $n > t$ such that $x_n \in \mathcal{A}$ and $x_n = x_t$ because $x_t$ is fixed when $x_t \notin \mathcal{A}$. Hence, whether $x_t \in \mathcal{A}$, there always exists $n \geq t$ such that $x_n \in \mathcal{A}$ such that $\norm{\nabla P(x_n)} \geq \eps$. However, this is a contradiction since  $\lim\limits_{t\in \mathcal{A}, t\rightarrow \infty} \norm{\nabla P(x_t)} = 0$. The proof is then completed.
\end{proof}

Now we have proved global convergence result of Algorithm \ref{MINIMAX-TRACE}. Then we analyze the iteration complexity of Algorithm \ref{MINIMAX-TRACE}. We first provide a refinement of Lemma \ref{lambda success step}.
\begin{lemma}\quad\label{lambda success step 2}
	For any $t \in \mathbb{N}$, if the trial step $s_t$ and dual variable $\lambda_t$ satisfy 
	\begin{equation}\label{lambda assumption2}
	\lambda_t \geq 2 (C_1 + \frac{1}{2}C_2 + \frac{1}{2}H_{Lip} + \eta)
	\norm{s_t}. 
	\end{equation}
	then $\norm{s_t} = \delta_t$ and $\rho_t \geq \eta$.
\end{lemma}
\begin{proof}\label{ass lambda2}
	By \eqref{lambda assumption2} we have $\lambda_t > 0$, then $\norm{s_t} = \delta_t$ which follows directly from \eqref{complementary condition} in Lemma \ref{TR sub}. Since $P(x)$ is a smooth function, there exists a point $\overline{x}_t \in [x_t,x_t + s_t]$ such that 
	\begin{equation}\label{gap between q(s_t) and P_t(x_t+s_t) 2}
	\begin{split}
	q_t(s_t) - P(x_t + s_t) &= P(x_t) + g_t^T s_t + \frac{1}{2} s_t^T H_t s_t - P(x_t + s_t) \\
	&= -\nabla P(x_t)^T s_t - \frac{1}{2} s_t^T \nabla^2 P(\overline{x}_t) s_t + g_t^T s_t + \frac{1}{2} s_t^T H_t s_t \\
	&\geq -\norm{g_t - \nabla P(x_t)} \norm{s_t} - \frac{1}{2} \norm{\nabla^2 P(\overline{x}_t) - \nabla^2 P(x_t)} \norm{s_t}^2 - \frac{1}{2} \norm{\nabla^2 P(x_t) - H_t} \norm{s_t}^2 \\
	&\geq -\norm{g_t - \nabla P(x_t)} \norm{s_t} - \frac{1}{2} H_{Lip} \norm{s_t}^3 - \frac{1}{2} \norm{\nabla^2 P(x_t) - H_t} \norm{s_t}^2.
	\end{split}
	\end{equation}
	Hence, Lemma \ref{gap between P(x_t) and q_t(s_t)}, Lemma \ref{bound lemma}, \eqref{lambda assumption2} and \eqref{gap between q(s_t) and P_t(x_t+s_t) 2} imply that, for any $t\in \mathbb{N}$,
	\begin{equation}
	\begin{split}
	P(x_t) - P(x_t + s_t) &= P(x_t) - q_t(s_t) + q_t(s_t) - P(x_t + s_t) \\
	&\geq \frac{\lambda_t}{2} \norm{s_t}^2 - (C_1 + \frac{1}{2}C_2 + \frac{1}{2} H_{Lip}) \norm{s_t}^3 \\
	&\geq \eta \norm{s_t}^3,
	\end{split}
	\end{equation}
	which implies that $\rho_t \geq \eta$. The proof is then completed.
\end{proof}



We now establish a lower bound on the norm of any accepted step in $\mathcal{A}_{\sigma}$.
\begin{lemma}\quad\label{s gradient}
	For all $t \in \mathcal{A}_{\sigma}$, the accepted step $s_t$ satifies
	\begin{equation}
	\norm{s_t} \geq (H_{Lip} + C_1 + C_2 + \sigma_{max})^{-\frac{1}{2}}
	\norm{\nabla P(x_{t+1})}^{\frac{1}{2}}.
	\end{equation} \nonumber
\end{lemma}
where $\sigma_{max}$ is an upper bound of $\sigma_t$. 
\begin{proof}
	Firstly, by Lemma 3.18 in \cite{TRACE}, we have $\sigma_t \leq \sigma_{max}$ for all $t\in \mathbb{N}$. For all $t \in \mathcal{A}_{\sigma}$, there exists $\overline{x}_t$ on the line segment $[x_t,x_t + s_t]$ such that
	\begin{equation}\nonumber
	\begin{split}
	\norm{\nabla P(x_{t+1})} &= \norm{\nabla P(x_{t+1}) - g_t - (H_t + \lambda_t I)s_t} \\
	&= \norm{\nabla P(x_{t+1}) - \nabla P(x_t) + \nabla P(x_t) - g_t - (H_t + \lambda_t I)s_t} \\
	&= \norm{[\nabla^2 P(\overline{x}_t) - \nabla^2 P(x_t) + \nabla^2 P(x_t) - H_t] s_t + \nabla P(x_t) - g_t - \lambda_t s_t} \\
	&\leq H_{Lip} \norm{s_t}^2 + C_2 \norm{s_t}^2 + C_1 \norm{s_t}^2 + \frac{\lambda_t}{\norm{s_t}} \norm{s_t}^2 \\
	&\leq (H_{Lip} + C_1 + C_2 + \sigma_{max}) \norm{s_t}^2, 
	\end{split}
	\end{equation}
	where the first equation holds by \eqref{first opt condition}, the first inequality holds by Lemma \ref{bound lemma} and the second inequality holds by $\sigma_t \leq \sigma_{max}$.
\end{proof}

We first give the upper bound of the number of accepted steps that may occur in iterations in which the norm of the gradient of the objective is above a prescribed positive threshold.
\begin{lemma}\quad \label{A sigma lemma}
	For a scalar $\eps \in (0,\infty)$, the total number of elements in the index set
	\[
	\mathcal{K}_{\eps} := \{
	t \in \mathbb{N}: t \geq 1, (t - 1) \in \mathcal{A}_{\sigma}, \norm{\nabla P(x_t)} > \eps
	\}
	\]
	is at most
	\begin{equation}
	\left\lfloor 
	\left(
	\frac{P(x_0) - P^*}{\eta (H_{Lip} + C_1 + C_2 + \sigma_{max})^{-\frac{3}{2}} }
	\right) \eps^{-\frac{3}{2}}
	\right\rfloor
	=: K_{\sigma}(\eps) \geq 0.
	\end{equation}
\end{lemma}
\begin{proof}
	By Lemma \ref{s gradient}, we have for all $t \in \mathcal{K}_{\eps}$ that
	\begin{equation}\nonumber
	\begin{split}
	P(x_{t-1}) - P(x_t) &\geq \eta \norm{s_{t-1}}^3 \\
	&\geq \eta (H_{Lip} + C_1 + C_2 + \sigma_{max})^{-\frac{3}{2}}
	\norm{\nabla P(x_t)}^{\frac{3}{2}} \\
	&\geq \eta (H_{Lip} + C_1 + C_2 + \sigma_{max})^{-\frac{3}{2}}
	\eps^{\frac{3}{2}}.
	\end{split}
	\end{equation}
	In addition, we have by Theorem \ref{gradient convergence theorem} that  $|\mathcal{K}_{\eps}| < \infty$. Hence, we have that
	\[
	P(x_0) - P^* \geq \sum_{t \in \mathcal{K}_{\eps}} (P(x_{t-1}) - P(x_t)) \geq |\mathcal{K}_{\eps}| \eta (H_{Lip} + C_1 + C_2 + \sigma_{max})^{-\frac{3}{2}}
	\eps^{\frac{3}{2}}.
	\]
	Rearranging this inequality to yield an upper bound for $|\mathcal{K}_{\eps}|$, we obtain the desired result.
\end{proof}

Now we can prove that the number of iterations required to find an $\eps$-FSP is at most $\mathcal{O}(\eps^{-3/2})$.
\begin{theorem}\quad \label{fsp theorem}
	Suppose Assumption \ref{basic assumption} and Assumption \ref{bound assumption} holds. For any given $\eps \in (0,\infty)$, the total number of elements in the index set
	\[
	\{t \in \mathbb{N}: \norm{\nabla P(x_t)} > \eps\}
	\]
	is at most
	\begin{equation}
	\mathcal{K}(\eps):= 1 + (\mathcal{K}_{\sigma}(\eps) + \mathcal{K}_{\Delta}) (1 + \mathcal{K}_{\mathcal{C}}),
	\end{equation}
	where $\mathcal{K}_{\sigma}(\eps)$ is defined in Lemma \ref{A sigma lemma}, $\mathcal{K}_{\Delta} = \left \lfloor
	\frac{P(x_0) - P^*}{\eta \Delta_0^3}
	\right \rfloor$ and $\mathcal{K}_{\mathcal{C}}=1 + \left\lfloor
	\frac{1}{\log(\min\{ \gamma_{\lambda},\gamma_{C}^{-1}
		\})}
	\log\left(
	\frac{\sigma_{max}}{\underline{\sigma}}
	\right)
	\right\rfloor$, respectively. Consequently, we have $\mathcal{K}(\eps) = \mathcal{O}(\eps^{-3/2})$.
\end{theorem}
\begin{proof}
	Without loss of generality, we may assume that at least one iteration is performed. Firstly, we partiton the set $\{t \in \mathbb{N}: \norm{\nabla P(x_t)} \geq \eps\}$ into four disjoint subsets:
 	$\{
 	t \in \mathbb{N}: t \geq 1, (t - 1) \in \mathcal{A}_{\sigma}, \norm{\nabla P(x_t)} > \eps
 	\}$,
 	$\{
 	t \in \mathbb{N}: t \geq 1, (t - 1) \in \mathcal{A}_{\Delta}, \norm{\nabla P(x_t)} \geq \eps
 	\}$,
 	$\{
 	t \in \mathbb{N}: t \geq 1, (t - 1) \in \mathcal{C}, \norm{\nabla P(x_t)} \geq \eps
 	\}$,
 	$\{
 	t \in \mathbb{N}: t \geq 1, (t - 1) \in \mathcal{E}, \norm{\nabla P(x_t)} \geq \eps
 	\}$.
 	By Lemma 3.21 in \cite{TRACE}, we have that the cardinality of the set $\mathcal{A}_{\Delta}$ is bounded above by 
 	\begin{equation} \label{A Delta}
 	\left \lfloor
 	\frac{P(x_0) - P^*}{\eta \Delta_0^3}
 	\right \rfloor = \mathcal{K}_{\Delta} \geq 0.
 	\end{equation}
	For the purpose of deriving an upper bound on the numbers of contraction and expansion iterations that may occur until the next accepted step, let us define, for any given $\hat{t} \in \mathcal{A}$,
	\begin{equation}\nonumber
	\begin{split}
	t_{\mathcal{A}}(\hat{t}) := \min \{
	t \in \mathcal{A}: t > \hat{t}
	\}, \\ 
	\mathcal{I}(\hat{t}) := \{
	t \in \mathbb{N}: \hat{t} < t < t_{\mathcal{A}}(\hat{t})
	\}.
	\end{split}
	\end{equation}
	For any $\hat{t} \in \mathbb{N}$, if $\hat{t} \in \mathcal{A} \cup \{0\}$, by Lemmas 3.22 and 3.24 in \cite{TRACE}, we have that
	\begin{equation} \label{expand bound}
		\mathcal{E} \cap \mathcal{I}(\hat{t}) \subseteq \{\hat{t} + 1\}
	\end{equation} 
	and
	\begin{equation} \label{contract bound}
	|\mathcal{C}\cap \mathcal{I}(\hat{t})|
	\leq 1 + \left\lfloor
	\frac{1}{\log(\min\{ \gamma_{\lambda},\gamma_{C}^{-1}
		\})}
	\log\left(
	\frac{\sigma_{max}}{\underline{\sigma}}
	\right)
	\right\rfloor = \mathcal{K}_{\mathcal{C}} \geq 0.
	\end{equation} 
	Lemma \ref{A sigma lemma} and \eqref{A Delta} guarantee that the total number of elements in the index set $\{t\in \mathbb{N}: t\geq 1, t-1\in \mathcal{A}, \norm{\nabla P(x_t)} \geq \eps\}$ is at most $\mathcal{K}_{\sigma}(\eps) + \mathcal{K}_{\Delta}$. By \eqref{contract bound} and \eqref{expand bound}, we have that at most $1+\mathcal{K}_{\mathcal{C}}$ sub-problem routine calls are required in expansion and contraction, which completes the proof.
\end{proof}

Now we have proved global convergence result of Algorithm \ref{MINIMAX-TRACE} and the iteration complexity of Algorithm \ref{MINIMAX-TRACE} to find $\eps$-FSP. Based on these results, we then analyze the iteration complexity of Algorithm \ref{MINIMAX-TRACE} to find $(\eps, \sqrt{\eps})$-SSP. We continue to use the notation introduced in Lemma \ref{radius lower bound}, where, for all $t \in \mathbb{N}$, we let $H_t = V_t \Xi_t V_t^T$ where $V_t$ is an orthonormal matrix of eigenvectors and $\Xi_t = diag(\xi_{t,1},\dots,\xi_{t,n})$ with $\xi_{t,1}\leq \dots \leq \xi_{t,n}$ is a diagonal matrix of eigenvalues of $H_t$.

%
\begin{theorem}\quad
	Under Assumption \ref{basic assumption} and Assumption \ref{bound assumption}, for a scalar $\eps \in (0,\infty)$, the total number of elements in the index set
	\[
	\{t \in \mathbb{N}: \norm{\nabla P(x_t)} > \eps\ \text{ or } \xi_{t,1} < -\frac{\sqrt{\eps}}{2} \}
	\]
	is at most
	\begin{equation}
	\mathcal{K}_2 (\eps):= 1 + (\mathcal{K}_{\sigma}(\eps) + 
	\mathcal{K}_{\sigma,2}(\eps) + 
	\mathcal{K}_{\Delta}) (1 + \mathcal{K}_{\mathcal{C}}),
	\end{equation}
	where $\mathcal{K}_{\sigma,2}(\eps) = \left\lfloor 
	\left(
	\frac{8(f_0 - f_{min})}{\eta  \sigma_{max}^{-3}}
	\right) \eps^{-3/2}
	\right\rfloor$, $\mathcal{K}_{\sigma}(\eps)$, $\mathcal{K}_{\Delta}$,  $\mathcal{K}_{\mathcal{C}}$
	are defined in Lemma \ref{A sigma lemma}, \eqref{A Delta} and \eqref{contract bound} respectively. Consequently, we have $\mathcal{K}_2(\eps) = \mathcal{O}(\eps^{-3/2})$. In addition, if we set $M_2=\sqrt{\eps}/2$ in Assumption \ref{bound assumption}, then $x_t$ is an  $(\eps, \sqrt{\eps})$-SSP for all $t > \mathcal{K}_2(\eps)$.
\end{theorem}

\begin{proof}
	By Lemma 3.27 in \cite{TRACE}, we have that the total number of elements in the index set
	\[
	\mathcal{K}_{\eps,2} := \{
	t \in \mathcal{A}_{\sigma}: \xi_{t,1} < -\frac{\sqrt{\eps}}{2}
	\}
	\]
	is at most
	\begin{equation}\label{A sigma 2}
	\left\lfloor 
	\left(
	\frac{8(P(x_0) - P^*)}{\eta  \sigma_{max}^{-3}}
	\right) \eps^{-3/2}
	\right\rfloor
	= K_{\sigma,2}(\eps) \geq 0.
	\end{equation}
	The rest of the proof is similar to that of Theorem \ref{fsp theorem}, except that we
	must now account for additional accepted steps—an upper bound for which is provided by \eqref{A sigma 2} that may occur until the leftmost eigenvalue of the Hessian matrix
	is above the desired threshold. For all $t > \mathcal{K}_2(\eps)$, we have $\norm{\nabla P(x_t)} \leq \eps$ and $\xi_{t,1} \geq -\frac{\sqrt{\eps}}{2}$. Then we obtain
	\begin{equation}
		\nabla^2 P(x_t) = H_t - (H_t - \nabla^2 P(x_t)) \succcurlyeq  -\frac{\sqrt{\eps}}{2} \mathbf I - \frac{\sqrt{\eps}}{2} \mathbf I = - \sqrt{\eps} \mathbf I,
	\end{equation} which show that $x_t$ is an  $(\eps, \sqrt{\eps})$-SSP.
\end{proof}

Now we have established convergence and complexity to $\eps$-FSP and $(\eps, \sqrt{\eps})$-SSP. In the rest of this section, we present the local superlinear
convergence of Algorithm \ref{MINIMAX-TRACE}. For this purpose, in addition to Assumptions \ref{basic assumption} and \ref{bound assumption},
we make the following assumption.
\begin{assumption}\label{local ass}
	Let $P(x)$ be twice Lipschitz continuously differentiable in a neighborhood of a point $x^*$ at which second-order sufficient conditions are satisfied ($\nabla P(x^*) = 0, \nabla^2 P(x^*) \succ 0$). Suppose the sequence $\{x_t\}$ converges to $x^*$.
\end{assumption}

\begin{lemma}\quad\label{local lemma}
	Under Assumption \ref{basic assumption}, Assumption \ref{bound assumption} and Assumption \ref{local ass}, there exists an iteration number $t_\mathcal{A} \in \mathbb{N}$ such that, for all $t \in \mathbb{N}$ with $t \geq t_{\mathcal{A}}$, the trial step, dual variable, and the iteration number satisfy $s_t = -H_t^{-1} g_t,  \lambda_t = 0$ and $t\in \mathcal{A}$, respectively. That is, eventually, all computed trial steps are Newton steps that are accepted by the algorithm.
\end{lemma}
\begin{proof}
	By Assumption \ref{local ass}, there exists a constant $\mu$ such that $\nabla^2 P(x_t) \succeq \mu I \succ 0$ for all sufficient large $t \in \mathbb{N}$. For sufficient large $t \in \mathbb{N}$, we have
	\begin{equation}\label{H_t local}
	H_t \succeq \nabla^2 P(x_t) - \norm{\nabla^2 P(x_t) - H_t} I   \succeq \mu I - C_2 \norm{s_t} I \succeq \frac{\mu}{2} I,
	\end{equation}
	where the second inequality holds by Lemma \ref{bound lemma} and the last inequality holds since $\norm{s_t} \rightarrow 0$, which implied by $\{x_t\} \rightarrow x^*$. Hence, for all such $t$, we either have $s_t = - H_t^{-1} g_t$ or the Newton step $-H_t^{-1} g_t$ lies outside the trust region.
	
	By Lemma \ref{gap between P(x_t) and q_t(s_t)}, we have that
	\begin{equation}\nonumber
	\begin{split}
	P(x_t) - q_t(s_t) &\geq \frac{1}{2} \norm{g_t} \min \left\{ \delta_t, \frac{\norm{g_t}}{\norm{H_t}}
	\right\} \\
	&\geq \frac{1}{2} \frac{\norm{s_t}}{\norm{H_t^{-1}}} \min \left\{ \norm{s_t}, \frac{\norm{s_t}}{\norm{H_t} \norm{H_t^{-1}}}
	\right\} \\
	&\geq \frac{1}{2} \frac{\norm{s_t}^2}{\norm{H_t} \norm{H_t^{-1}}^2}.
	\end{split}
	\end{equation}
	
	By Lemma \ref{bound lemma} and $x_t \rightarrow x*$, we have
	\begin{equation}
	\begin{split}
	\norm{H_t - \nabla^2 P(x^*)} 
	&\leq \norm{H_t - \nabla^2 P(x_t)} + \norm{\nabla^2 P(x_t) - \nabla^2 P(x^*)} \\ &\leq C_2 \norm{s_t} + H_{Lip} \norm{x_t - x^*} \rightarrow 0.
	\end{split}
	\end{equation}
	Thus, with $\xi^* = 1/(4\norm{\nabla^2 P(x^*)} \norm{\nabla^2 P(x^*)^{-1}}^2)$, we have for sufficiently large $t \in \mathbb{N}$ that
	\[
	P(x_t) - q_t(s_t) \geq \xi^* \norm{s_t}^2.
	\]
	
	As in the proof of Lemma \ref{lambda success step 2}, for all sufficiently large $t \in \mathbb{N}$, 
	\begin{equation}
	\begin{split}
	P(x_t) - P(x_t + s_t) &= P(x_t) - q_t(s_t) + q_t(s_t) - P(x_t + s_t) \\
	&\geq \xi^* \norm{s_t}^2 - (C_1 + \frac{1}{2}C_2 + \frac{1}{2} H_{Lip}) \norm{s_t}^3 \\
	&\geq \eta \norm{s_t}^3,
	\end{split}
	\end{equation}
 which means that for all sufficiently large $t \in \mathbb{N}$, we have $\rho_t \geq \eta$ such that $t \in \mathcal{A} \cup \mathcal{E}$. 
	
	By Lemma \ref{radius relation}, it follows that there exists a scalar constant $\delta_{min} > 0$ such that $\delta_t \geq \delta_{min}$ for all sufficiently large $t \in \mathbb{N}$. Since $\norm{s_t} \rightarrow 0$, we have $\norm{s_t} < \delta_{min} \leq \delta_t$ and $\lambda_t = 0$ for all sufficiently large $t \in \mathbb{N}$. This implies that $t\in \mathcal{A}$ and the Newton step lies in the trust region.
\end{proof}

\begin{theorem}\quad
	Under Assumption \ref{basic assumption}, Assumption \ref{bound assumption} and Assumption \ref{local ass}, it holds that 
	\begin{equation}
	\norm{\nabla P(x_{t+1})} = \mathcal{O}(\norm{\nabla P(x_t)}^2),
	\end{equation}
	i.e., the sequence $\{\nabla P(x_t)\}$ Q-quadratically converges to $0$.
\end{theorem}
\begin{proof}
	By \eqref{first opt condition}, we have 
	\begin{equation}\label{local ineq1}
		\begin{split}
			\norm{\nabla P(x_{t+1})} &= \norm{\nabla P(x_{t+1}) - g_t - H_t s_t} \\
			&= \norm{\nabla P(x_{t+1}) - \nabla P(x_t) - \nabla^2 P(x_t) s_t + \nabla P(x_t) - g_t + (\nabla^2 P(x_t) - H_t) s_t} \\
			&\leq \norm{\nabla P(x_{t+1}) - \nabla P(x_t) - \nabla^2 P(x_t) s_t} + \norm{\nabla P(x_t) - g_t} + \norm{(\nabla^2 P(x_t) - H_t) s_t} \\
			&\leq  ( \frac{H_{Lip}}{2} + C_1 + C_2 ) \norm{s_t}^2,  
		\end{split}
	\end{equation}
	where the last inequality is by Lemma \ref{lem:2.3} and Lemma \ref{bound lemma}.
	
	As in the proof of Lemma \ref{local lemma}, we have $s_t = -H_t^{-1} g_t$ and $H_t \succeq \frac{\mu}{2} \mathbf{I}$ for $t\geq t_{\mathcal{A}}$. By Lemma \ref{bound lemma} and the fact that $(a + b)^2 \leq 2(a^2 + b^2)$, we have
	\begin{equation}
		\norm{s_t}^2 \leq \norm{H_t^{-1}}^2 \norm{g_t}^2 \leq \frac{4}{\mu^2} 
		(\norm{g_t - \nabla P(x_t)} + \norm{\nabla P(x_t)})^2 
		\leq \frac{8}{\mu^2} 
		(C_1^2 \norm{s_t}^4 + \norm{\nabla P(x_t)}^2).
	\end{equation}
	Since $\norm{s_t} \rightarrow 0$, then for sufficient large $t$, we have
	\begin{equation}\label{local ineq2}
		\frac{1}{2} \norm{s_t}^2 \leq (1 - \frac{8}{\mu^2} C_1^2 \norm{s_t}^2) \norm{s_t}^2 \leq \frac{8}{\mu^2}\norm{\nabla P(x_t)}^2.
	\end{equation}   
	Plugging \eqref{local ineq2} into \eqref{local ineq1}, we have
	\begin{equation}
		\norm{\nabla P(x_{t+1})} \leq \frac{16}{\mu^2} ( \frac{H_{Lip}}{2} + C_1 + C_2 ) \norm{\nabla P(x_t)}^2,
	\end{equation}
	which completes the proof.
\end{proof}

\section{Numerical Experiments}
In this section,  motivated by \cite{du2017gradient} and \cite{huang2022efficiently}, we compare our algorithms with the GDA algorithm \cite{GDA} and the MCN algorithm \cite{luocubic} for solving the following minimax optimization problem:
\begin{equation}\label{numerical problem1}
\min \limits_{x\in \mathbb{R}^n}\max \limits_{y\in \mathbb{R}^m}g(x) - \frac{1}{2} y^2.
\end{equation}
where
\begin{equation}
g(x) = \left\{
\begin{array}{ll}
g_{i,1}(x) & x_1, ..., x_{i-1}\in [2\tau, 6\tau],\ x_i\in [0,\tau],\ x_{i+1},..., x_n\in [0,\tau],\ 1\leq i\leq n-1,\\
g_{i,2}(x) &  x_1, ..., x_{i-1}\in [2\tau, 6\tau],\ x_i\in [\tau, 2\tau],\ x_{i+1},..., x_n\in [0,\tau],\ 1\leq i\leq n-1,\\
g_{n,1}(x) & x_1, ..., x_{n-1}\in [2\tau, 6\tau],\ x_n\in [0,\tau], \\
g_{n,2}(x) & x_1, ..., x_{n-1}\in [2\tau, 6\tau],\ x_n\in [\tau, 2\tau], \\
g_{n+1, 1}(x) & x_1, ..., x_{n}\in [2\tau, 6\tau],
\end{array}
\right.
\end{equation}
with
\begin{align}
g_{i,1}(x) &= \sum_{j=1}^{i-1}L(x_j - 4\tau)^2 -\gamma x_i^2 + \sum_{j=i+1}^{n}Lx_j^2 - (i-1)\nu,\ 1\leq i\leq n-1, \\
g_{i,2}(x) &= \sum_{j=1}^{i-1}L(x_j - 4\tau)^2 + h(x) + \sum_{j=i+2}^{n}Lx_j^2 - (i-1)\nu,\ 1\leq i \leq n-1, \\
g_{n,1}(x) &= \sum_{j=1}^{n-1}L(x_j - 4\tau)^2 -\gamma x_n^2 - (n-1)\nu, \\
g_{n,2}(x) &= \sum_{j=1}^{n-1}L(x_j - 4\tau)^2 + h(x) - (n-1)\nu, \\
g_{n+1, 1}(x) &= \sum_{j=1}^{n}L(x_j - 4\tau)^2 - n\nu,
\end{align}
and
\begin{align}
h(x) &= \left\{
\begin{array}{ll}
h_1(x_i) + h_2(x_i)x_{i+1}^2 & x_1, ..., x_{i-1}\in [2\tau, 6\tau],\ x_i\in [\tau, 2\tau],\ x_{i+1},..., x_n\in [0,\tau],\\ 
&1\leq i\leq n-1,\\
h_1(x_n) & x_1, ..., x_{n-1}\in [2\tau, 6\tau],\ x_n\in [\tau, 2\tau],\\
0 & \text{otherwise,}
\end{array}
\right. \\
h_1(x) &= -\gamma x^2 + \frac{(-14L + 10\gamma)(x-\tau)^3}{3\tau} + \frac{(5L - 3\gamma)(x-\tau)^4}{2\tau^2},\\
h_2(x) &= -\gamma - \frac{10(L+\gamma)(x-2\tau)^3}{\tau^3} - \frac{15(L+\gamma)(x-2\tau)^4}{\tau^4} - \frac{6(L+\gamma)(x-2\tau)^5}{\tau^5},
\end{align}
and
\[
L > 0,\ \gamma >0,\ \tau = e,\ \nu = -h_1(2\tau) + 4L\tau^2.
\]
Note that $g(x)$ is only defined on the following domain:
\begin{equation}
D_0 = \bigcup_{i=1}^{n+1}\left\{x\in \mathbb{R}^n: 6\tau \geq x_1,...,x_{i-1}\geq 2\tau, 2\tau\geq x_i\geq 0, \tau\geq x_{i+1},...,x_n\geq 0\right\}.
\end{equation}
By Lemma A.3 in \cite{du2017gradient}, we know that there is only one local optimum, i.e., $(4\tau,...,4\tau)^{\top}$ and $d$ stationary points, i.e.,
\[
(0,...,0)^{\top}\ ,(4\tau, 0,..., 0)^{\top}\ ,..., (4\tau, ..., 4\tau,0)^{\top}.
\]
\begin{figure}[t]
	\centering 
	\includegraphics[scale=0.25]{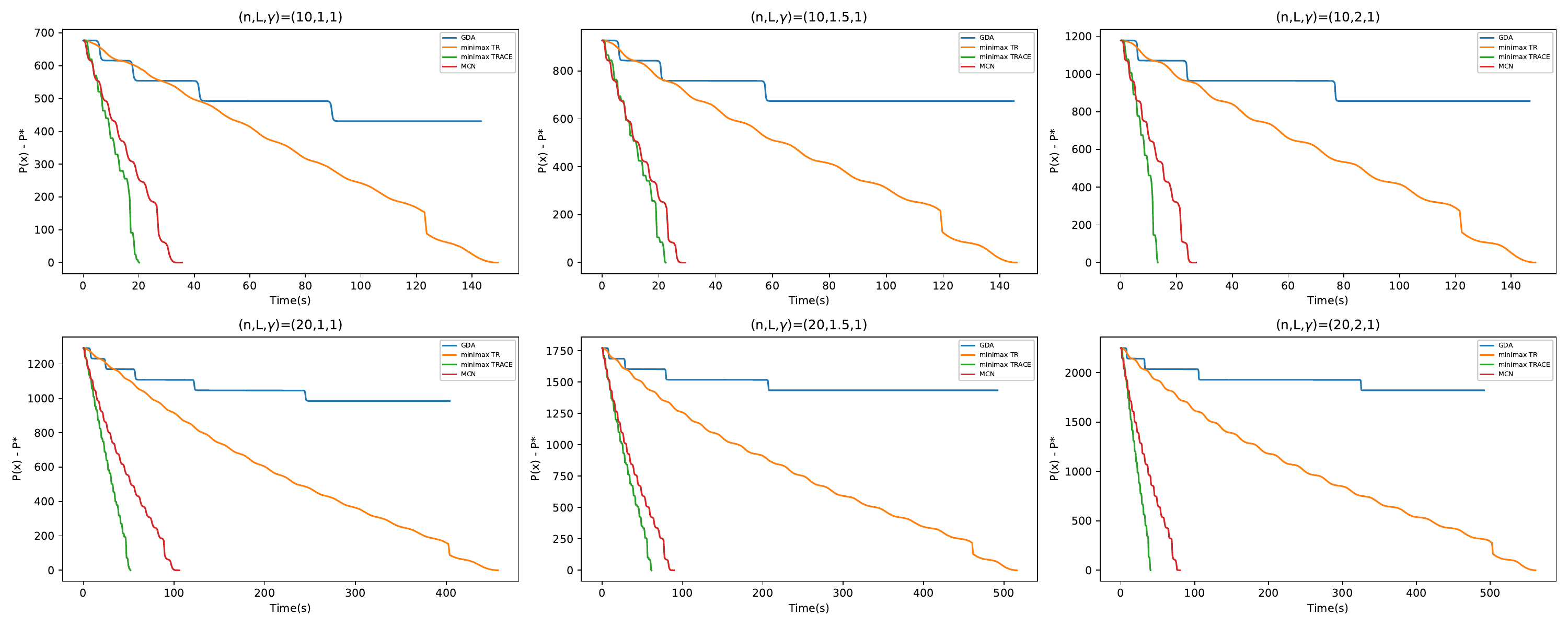}	
	\caption{Numerical results of four tested algorithms for solving \eqref{numerical problem1} with different choices of $n$, $L$ and $\gamma$.}
	\label{problem1 fig}
\end{figure}
We set the dimension of variable $y$ to be $5$ and set $x_0$ to be $(10^{-3},...,10^{-3})^{\top}$ which is close to the stationary point $(0,...,0)^{\top}$. $y_0$ is a random vector. We compare the running time against $P(x)-P^*$ for the four tested algorithms. In Figure \ref{problem1 fig}, the horizontal segment indicates that GDA gets stuck around some stationary points which are not local optimum, while all the second-order algorithms can escape them. MINIMAX-TR converges slowly since the trust region radius is set to be a small value. MINIMAX-TRACE outperforms the state-of-the-art MCN algorithm.

\section{Conclusions}
In this paper, we propose a Minimax Trust Region (MINIMAX-TR) algorithm and a Minimax Trust Region algorithm with Contractions and Expansions (MINIMAX-TRACE) for solving nonconvex-strongly concave minimax problems. Both algorithms can find 
an $(\eps, \sqrt{\eps})$-SSP within $\mathcal{O}(\eps^{-1.5})$ number of iterations. Moreover, MINIMAX-TRACE has locally quadratic convergence.

\end{document}